\documentclass{amsart}
\usepackage{amscd,amssymb,amsopn,amsmath,amsthm,graphics,amsfonts,accents,enumerate,verbatim,calc}
\usepackage[dvips]{graphicx}
\usepackage[colorlinks=true,linkcolor=red,citecolor=blue]{hyperref}
\usepackage[all]{xy}
\usepackage{cite}
\usepackage{tikz}
\usepackage{tikz-cd}
\usepackage{mathrsfs}

\calclayout

\newcommand{\rt}{\rightarrow}
\newcommand{\lrt}{\longrightarrow}

\newcommand{\va}{\varphi}
\newcommand{\st}{\stackrel}

\newcommand{\La}{\Lambda}
\newcommand{\Ga}{\Gamma}
\newcommand{\Om}{\Omega}

\newcommand{\Z}{\mathbb{Z}}

\newcommand{\SA}{\mathscr{A}}
\newcommand{\SB}{\mathscr{B}}
\newcommand{\SC}{\mathscr{C}}
\newcommand{\SE}{\mathscr{E}}
\newcommand{\SF}{\mathscr{F}}

\newcommand{\SI}{\mathscr{I}}
\newcommand{\SJ}{\mathscr{J}}

\newcommand{\SM}{\mathscr{M}}

\newcommand{\SX}{\mathscr{X}}

\newcommand{\CI}{\mathcal{I} }

\newcommand{\X}{\mathbf{X}}
\newcommand{\Y}{\mathbf{Y}}

\newcommand{\Ob}{{\rm{Ob}}}

\newcommand{\mmod}{{\rm{{mod\mbox{-}}}}}

\newcommand{\Gprj}{{\Gp\mbox{-}}}

\newcommand{\PB}{{\rm{PB}}}

\newcommand{\op}{{\rm{op}}}

\newcommand{\cone}{{\rm{cone}}}

\newcommand{\Gp}{{\rm{Gprj}}}

\newcommand{\Hom}{{\rm{Hom}}}
\newcommand{\Ext}{{\rm{Ext}}}
\newcommand{\Yext}{{\rm{Yext}}}

\theoremstyle{plain}
\newtheorem{theorem}{Theorem}[section]
\newtheorem{corollary}[theorem]{Corollary}
\newtheorem{lemma}[theorem]{Lemma}

\newtheorem{proposition}[theorem]{Proposition}

\newtheorem{notation}[theorem]{Notation}

\theoremstyle{definition}
\newtheorem{definition}[theorem]{Definition}
\newtheorem{example}[theorem]{Example}

\newtheorem{construction}[theorem]{Construction}

\newtheorem{remark}[theorem]{Remark}
\newtheorem{setup}[theorem]{Setup}
\newtheorem{s}[theorem]{}

\theoremstyle{plain}

\theoremstyle{definition}

\numberwithin{equation}{section}

\begin{document}

\title[Higher Ideal Approximation Theory]{Higher Ideal Approximation Theory}

\author[Javad Asadollahi and Somayeh Sadeghi]{Javad Asadollahi and Somayeh Sadeghi}

\address{Department of Pure Mathematics, Faculty of Mathematics and Statistics, University of Isfahan, P.O.Box: 81746-73441, Isfahan, Iran}
\email{asadollahi@ipm.ir, asadollahi@sci.ui.ac.ir}
\email{so.sadeghi@sci.ui.ac.ir }



\begin{abstract}
Let $\SC$ be an $n$-cluster tilting subcategory of an exact category $(\SA, \SE)$, where $n \geq 1$ is an integer. It is proved by Jasso that if $n> 1$, then $\SC$ although is no longer exact, but has a nice structure known as $n$-exact structure. In this new structure conflations are called admissible $n$-exact sequences and are $\SE$-acyclic complexes with $n+2$ terms in $\SC$. Since their introduction by Iyama, cluster tilting subcategories has gained a lot of traction, due largely to their links and applications to many research areas, many of them unexpected. On the other hand, ideal approximation theory, that is a gentle generalization of the classical approximation theory and deals with morphisms and ideals instead of objects and subcategories, is an active area that has been the subject of several researches. Our aim in this paper is to introduce the so-called `ideal approximation theory' into `higher homological algebra'. To this end, we introduce some important notions in approximation theory into the theory of $n$-exact categories and prove some results. In particular, the higher version of the notions such as ideal cotorsion pairs, phantom ideals, Salce's Lemma and Wakamatsu's Lemma for ideals will be introduced and studied. Our results motivate the definitions and show that $n$-exact categories are the appropriate context for the study of `higher ideal approximation theory'.
\end{abstract}

\maketitle

\tableofcontents

\section{Introduction}
One of the cornerstones of the modern representation theory is the Auslander's theorem \cite{Au} proving that there is a bijective correspondence between Morita equivalence classes of Artin algebras $\La$ of finite representation type and Morita equivalence classes of algebras $\Ga$ satisfying homological conditions that global dimension of $\Ga$ is less than or equal to two and its dominant dimension \cite{Ta} is greater or equal than two.

In a successful attempt to build up a higher version of Auslander's correspondence and also generalizing the classical theory of almost split sequences of Auslander-Reiten, Iyama \cite{I1, I2, I3} introduced the notion of $n$-cluster tilting subcategories, where $n$ is an integer greater or equal than $1$. Soon it is realized that these subcategories play a crucial role in the theory and so cluster tilting subcategories became the subject of several researches.

In particular, study of the structure of such subcategories leads Jasso \cite{Ja} to a higher version of the classical homological algebra and as a consequence new notions such as $n$-abelian and $n$-exact categories were born. These notions provide appropriate higher version of the classical abelian and exact categories, in the sense that $1$-abelian and $1$-exact categories are the usual abelian and exact categories. Instead of the usual kernels and cokernels, resp. inflations and deflations, in these categories we have the notions of $n$-kernels and $n$-cokernels and the role of short exact sequences are played by exact complexes with $n+2$ terms.

Since their introduction, various authors studied further several properties of $n$-abelian and $n$-exact categories, also trying to introduce notions of classical homological algebra into this higher version, which is now known as higher homological algebra. For some attempts in this direction see e.g. \cite{Jor}, \cite{JK}, \cite{AHS} and \cite{Kv}.

Considering the fact that approximation theory is one of the efficient tools for studying complicated objects in a category, the aim of the present paper is to contribute to this project by introducing some notions of ideal approximation theory into higher homological algebra.

The starting point of approximation theory is the discovery of the existence of injective envelopes by Baer in 1940. Approximation theory, that is approximation of complicated objects of a category by simpler objects in a specific subcategory, is essentially based on the notions of preenvelopes and precovers. Recall that a class $\SF$ of $R$-modules is precovering if for every $R$-module $M$, there exists a morphism $\va: F \lrt M$ with $F \in \SF$ such that the induced morphism $\Hom_R(F', F) \lrt \Hom_R(F', M)$ is surjective, for all $F' \in \SF$. Dually the notion of (pre-) enveloping classes is defined. An important problem in this context is to investigate whether a class of modules is (pre-) enveloping or/and (pre)covering. For instance, flat cover conjecture, posed by Enochs \cite{E} states that the class of of flat modules is a precovering (and hence a covering) class. This conjecture proved affirmatively after about 30 years \cite{BEE}.

Thanks to the researches of Auslander and his colleagues for artin algebras, in particular, those of Auslander, Reiten and Smal{\o}, \cite{AS, AS2, AR}, approximation theory also plays a central role in representation theory of algebras under the name of left approximations (pre-envelopings) and right approximations (precoverings). For a good account on approximation theory see the monograph \cite{GT}.

In classical approximation theory, approximation is done by objects from a subcategory. But a nice generalization of the classical approximation theory, known as ideal approximation theory is studied systematically in \cite{FGHT} that  give morphisms and ideals of categories equal importance as objects and subcategories. In this theory, the role of the objects and subcategories in classical approximation theory is replaced by morphisms and ideals of the category.

An ideal of a category is an additive subfunctor of the Hom functor, which is closed under compositions by morphisms from left and right. For instance, the phantom ideal and phantom cover in module category are studied extensively by Herzog in \cite{H1}. Moreover, continuing the idea of \cite{FGHT}, Fu and Herzog in \cite{FH} studied ideal versions of some pillars of classical approximation theory, such as cotorsion pairs, Salce's Lemma and Wakamatsu's Lemma. Very recently, Breaz and Modoi \cite{BM} studied ideal cotorsion pairs in extension closed subcategories of triangulated categories. For more research in this direction see \cite{EAO} and \cite{O}.

Following these ideas, the general goal of this paper is to introduce ideal approximation theory into the higher homological algebra. Our results show that the correct context in which to carry these arguments out is that of an $n$-cluster tilting subcategory of an exact category. By \cite[\S 4]{Ja} we know that these subcategories are $n$-exact, i.e. with `admissible' sequences with $n+2$ terms as conflations. Using this structure, a `higher ideal approximation theory' is developed in this paper. We state and prove some foundational results in this subject to motivate the theory.

The paper is structured as follows. In section \ref{Higher Cotorsion}, we introduce and study a higher version of the notion of ideal cotorsion pairs \cite[Definition 12]{FGHT}. For a good account for cotorsion pairs in abelian categories see \cite[\S V.3]{BR}. In section \ref{Higher Phantom}, phantom morphisms in $n$-cluster tilting subcategories of exact categories will be studied. It is a higher version of the notion of phantom morphisms studied by Herzog \cite{H1, H2}. As an example, we study $n$-pure phantom morphisms. We recall that phantom morphisms were introduced by McGibbon \cite{Mc} for use in topology to study maps between CW-complexes. Neeman \cite{N} introduced this concept into the context of triangulated categories. The theory also was developed in the stable category of a finite group ring in a series of works of Benson and Gnacadja \cite{BG1, BG2, G}.

Higher version of Salce's Lemma for $n$-ideal cotorsion pairs is treated in section \ref{Higher Salce}. Salce's Lemma \cite{S} is one of the main theorems in the classical approximation theory. It relates the notions of (special) precoverings, (special) preenvelopings and cotorsion pairs. By introducing an interesting exact structure on the morphism category of an exact category, called ME-exact structure, an ideal version of Salce's Lemma is proved in \cite[Theorem 6.3]{FH}.

The main purpose of section \ref{Special Precovering} is to study connections between special precovering ideals and $n$-phantom morphisms in an $n$-cluster tilting subcategory. In particular, we show that every special precovering ideal, under some conditions, can be represented as an ideal of $n$-$\SF$-phantom morphisms, for some bifunctor $\SF$ of $\Ext^n$.

In the last section of the paper, we state and prove an ideal version of the Wakamatsu's Lemma \cite{Wak} in higher homological algebra. Note that an ideal version of Wakamatsu's Lemma in an exact category is proved in \cite[Theorem 10.3]{FH}. Moreover a version of Wakamatsu's Lemma for $(n+2)$-angulated categories is formulated and proved in \cite[\S 3]{Jor}.

\section{Higher homological algebra}
In this section we collect some basic facts and backgrounds on higher homological algebra we need throughout the paper. We are mainly work in an exact category $(\SA, \SE)$, where $\SA$ is an additive category and $\SE$ is a class of composable pairs (also called kernel-cokernel pairs) of morphisms in $\SA$ which is closed under isomorphisms and satisfies axioms of Definition 2.1 of \cite{B}. The composable pair $(i, p)$ in $\SE$ is denoted by $A' {\rightarrowtail} A {\twoheadrightarrow} A''$, while $i: A' \rt A$ is called an $\SE$-admissible monic and $p: A \rt A''$ is called an $\SE$-admissible epic. When the exact structure $\SE$ is clear from the context we just say admissible instead of $\SE$-admissible. Admissible pairs, admissible monics and admissible epics also called conflations, inflations and deflations, respectively, see \cite{Q} and \cite{K}.  For definitions and properties of exact categories we refer to \cite{B}.

\s{\sc $n$-exact categories.}
Let $n \geq 1$ be a fixed integer. The notion of $n$-exact categories is defined by Jasso in \cite[\S 4]{Ja} as a natural generalization of exact categories. Its idea appeared naturally in a series of papers by Iyama, studying a higher version of the Auslander's correspondence \cite{I1, I2, I3} and then axiomatized by Jasso \cite{Ja}. Let us recall the definitions and basic properties.

Let $\SC$ be an additive category. Let $f^0 : X^0 \lrt X^1$  be a morphism in $\SC$. An $n$-cokernel of $f^0$ is a sequence
\[X^1 \st{f^1}{\lrt} X^2 \lrt \cdots \lrt X^{n} \st{f^n}{\lrt} X^{n+1}\]
of morphisms in $\SC$ such that for every $X \in \SC$ the induced sequence
\[0 \lrt \SC(X^{n+1}, X) \st{f^n _*}{\lrt} \cdots \st{f^{ 1 }_*}{\lrt} \SC(X^1, X) \st{f^0_*}{\lrt} \SC(X^{0}, X)\]
of abelian groups is exact. Here and throughout we write $\SC( - , - )$ instead of $\Hom_{\SC}( - , -)$. We denote the $n$-cokernel of $f^0$ by $(f^1,  f^2, \cdots,  f^{n})$. The notion of $n$-kernel of a morphism $f^n:X^n \lrt X^{n+1}$ is defined similarly, or rather dually.

A sequence $X^0 \st{f^0 }{\lrt } X^1 \lrt \cdots \lrt X^n \st{f^n}{\lrt} X^{n+1}$ of objects and morphisms in $\SC$, is called $n$-exact \cite[Definitions 2.2, 2.4]{Ja} if $(f^0,  f^1, \cdots,  f^{n-1})$ is an $n$-kernel of $f^n$ and $(f^1, f^2, \cdots, f^{n})$ is an $n$-cokernel of $f^0$. An $n$-exact sequence like the above one, usually will be denoted by
\[ X^0 \st{f^0 }{ \rightarrowtail } X^1 \st{f^1}{\lrt}  X^2 \lrt \cdots \lrt X^n \st{ f^n }{\twoheadrightarrow } X^{n+1} .\]

Consider the complex
\[{\mathbf{X}}: \  X^0 \st{f^0}{\lrt} X^1 \lrt \cdots \lrt X^{n-1} \st{f^{n-1}}{\lrt} X^{n}\]
and morphism $g^0:X^0 \lrt Y^0$ in $\SC$. An $n$-pushout diagram of ${\bf X}$ along $g^0$ is a morphism
\begin{equation}
\begin{tikzcd}
X^0 \rar \dar{g^0} & X^1\rar \dar{g^1} & \cdots \rar & X^n \dar\\
Y^0  \rar & Y^1 \rar & \cdots \rar & Y^n
\end{tikzcd}
\end{equation}
of complexes such that in the mapping cone ${\bf C}=\cone(g)$
\[X^0 \st{d^{-1}_{\bf C}}{\lrt} X^1\oplus Y^0 \st{d^0_{\bf C}}{\lrt} \cdots \st{d^{n-2}_{\bf C}}{\lrt} X^n\oplus Y^{n-1} \st{d^{n-1}_{\bf C}}{\lrt} Y^n \]
the sequence $(d^0_{\bf C}, d^{1}_{\bf C}, \cdots, d^{n-1}_{\bf C})$ is an $n$-cokernel of $d^{-1}_{\bf C}.$ The maps in the mapping cone are defined as usual. For more details and properties of the $n$-pushout, and also $n$-pullback diagrams, see Subsection 2.3 of \cite{Ja}.

A morphism $f : {\X} \lrt {\Y}$ of $n$-exact sequences is called a weak isomorphism if $f^k$ and $f^{k+1}$ are isomorphisms for some $k \in \{ 0, 1, \cdots, n + 1\}$, where we set $n+2 := 0$.

An $n$-exact structure on $\SC$ is a class $\SX$ of $n$-exact sequences
\begin{equation*}
\begin{tikzcd}
 & \eta: \ X^0\ar[tail]{r}{f^0} & X^1 \ar{r}&\cdots \ar{r} & X^n \ar[two heads]{r}{f^n} & X^{n+1}
\end{tikzcd}
\end{equation*}
in $\SC$, called $\SX$-admissible $n$-exact sequences, that is closed under weak isomorphisms and satisfies the following axioms. $f^0$, resp. $f^n$, is then called an $\SX$-admissible monomorphism, resp. an $\SX$-admissible epimorphism. We just write admissible instead of $\SX$-admissible, when the class $\SX$ is clear from the context.
\begin{itemize}
\item [$(E0)$] The sequence
 $0\rightarrowtail 0 \rightarrow \cdots\rightarrow 0\twoheadrightarrow 0$ is an admissible $n$-exact sequence.
\item [$(E1)$] The class of admissible monomorphisms is closed under composition.
\item [$(E2)$] The $n$-pushout  of an admissible $n$-exact sequence $(d_X^0, d_X^1, ..., d_X^n)$  along  morphism $f : X^0 \rt Y^0$ exists and $d_Y^0$ is  an admissible monomorphism.
\begin{equation}\label{push-out}
\begin{tikzcd}
X^0 \rar[tail]{d^0_X} \dar{f}& X^1\rar{d_X^1} \dar& \cdots\rar & X^{n-1} \rar{d^{n-1}_X} \dar & X^n\rar[two heads]{d^n_X}\dar& X^{n+1}\\
Y^0  \rar{d^0_Y}& Y^1\rar& \cdots \rar & Y^{n-1} \rar{d^{n-1}_Y} & Y^n
\end{tikzcd}
\end{equation}
\item [$(E1^{\op})$] The class of admissible epimorphisms is closed under composition.
\item[$(E2^{\rm op})$] The $n$-pullback of an admissible $n$-exact sequence $(d_X^0, d_X^1, ..., d_X^n)$  along morphism $g : Y^{n+1} \rt X^{n+1}$ exists and $d_Y^n$ is an admissible epimorphism.
\begin{equation}\label{pull-back}
\begin{tikzcd}
& Y^1  \rar{d^1_Y}\dar & Y^2 \rar{d^2_Y} \dar & \cdots \rar &Y^n\rar{d^{n}_Y}\dar & Y^{n+1}  \dar{g}\\
X^0 \rar[tail]{d^0_X} & X^1\rar{d_X^1} & X^2 \rar{d^2_X} & \cdots\rar & X^n\rar[two heads]{d^n_X}& X^{n+1}\
\end{tikzcd}
\end{equation}
\end{itemize}
An $n$-exact category is a pair $(\SC, \SX)$ where $\SC$ is an additive category and $\SX$ is an $n$-exact structure on $\SC$.

Let $(\SC, \SX)$ be an $n$-exact category. By \cite[Proposition 4.8]{Ja}, the $n$-pushout diagram \ref{push-out} can be completed to the commutative diagram
\begin{equation}\label{completed pushout}
\begin{tikzcd}
X^0 \rar[tail]{d^0_X} \dar{f}& X^1\rar{d_X^1} \dar & \cdots\rar  & X^n\rar[two heads]{d^n_X} \dar & X^{n+1} \dar[equals]\\
Y^0  \rar[tail]{d^0_Y}& Y^1\rar & \cdots\rar{d^{n-1}_Y}& Y^n \rar[two heads, dotted]{d^n_Y}  & X^{n+1}
\end{tikzcd}
\end{equation}
where the lower row also is an admissible $n$-exact sequence. Similar result holds true for the $n$-pullback diagrams \ref{pull-back}. For more details and properties of $n$-exact categories see \cite[Section 4]{Ja}.

The most known examples of $n$-exact categories are $n$-cluster tilting subcategories of exact categories \cite[Theorem 4.14]{Ja}, introduced by Iyama in a series of papers, see e.g. \cite{I1, I2}. Let us recall their definitions and some basic properties we need later in the paper.

\s
Let $(\SA, \SE)$ be an exact category. A morphism $A \lrt A'$ in $\SA$ is called proper if it admits a factorization $A \twoheadrightarrow K \rightarrowtail A'$, such that $A \twoheadrightarrow K$ is an $\SE$-admissible epic and $K \rightarrowtail A'$ is an $\SE$-admissible monic. A sequence
\begin{equation*}
\begin{tikzcd}
\cdots \rar & A^{i-1} \ar{rr} \drar[two heads] && A^i \ar{rr} \drar[two heads] && A^{i+1} \rar & \cdots\\
&& K^{i-1} \urar[tail] && K^i \urar[tail]
\end{tikzcd}
\end{equation*}
of proper morphisms is called an $\SE$-acyclic complex if for all $i \in \Z$ the sequence $K^{i-1} \rightarrowtail A^i \twoheadrightarrow K^i$ is an $\SE$-admissible exact sequence. An $\SE$-acyclic complex such as
\[\delta: A' \rightarrowtail A^1 \rt A^2 \rt \cdots \rt A^i \twoheadrightarrow A \]
is called an $\SE$-acyclic complex of length $i$ with end-terms $A$ and $A'$.

\s \label{Yoneda-Ext} Let $(\SA, \SE)$ be an exact category. Let
\[\delta: A' \rightarrowtail B^1 \rt B^2 \rt \cdots \rt B^i \twoheadrightarrow A \]
be an $\SE$-acyclic complex of length $i$ with end-terms $A$ and $A'$. A morphism between $\delta$ and $\delta'$, denoted by $\delta \lrt \delta'$, is a commutative diagram
\begin{equation*}
\begin{tikzcd}
\delta: & A' \rar[tail] \dar[equal] & A^1  \rar \dar & A^2 \rar \dar & \cdots \rar & A^i \rar[two heads] \dar & A  \dar[equal]\\
\delta': & A' \rar[tail] & B^1\rar & B^2 \rar & \cdots \rar & B^i \rar[two heads] & A
\end{tikzcd}
\end{equation*}
We say that $\delta$ and $\delta'$ are Yoneda equivalent if there is a sequence of morphisms
\[\delta_0=\delta \lrt \delta_1 \longleftarrow \delta_2 \lrt \cdots \lrt \delta_{t-1} \longleftarrow \delta_t=\delta'\]
of $\SE$-acyclic complexes of length $i$ with end-terms $A$ and $A'$. It is known that the collection of Yoneda equivalence classes of all $\SE$-acyclic complexes of length $i$ with the same end-terms form a (big) abelian group under Baer sum. For a detailed explanation, see Subsection 6.2 of \cite{FS}. For objects $A$ and $A'$ in $\SA$, this group is denoted by $\Ext^i_{\SE}(A, A')$. Note that in general $\Ext^i_{\SE}(A, A')$ is not a set \cite[Remark 6.2 and Remark 6.20]{FS}. But if $\SA$ is a small category or contains enough $\SE$-projective or enough $\SE$-injective objects, then $\Ext^i_{\SE}(A, A')$ will be a set \cite[Remark 6.44]{FS}.

\begin{definition}(\cite[Definition 4.13]{Ja})
Let $(\SA, \SE)$ be a small exact category. A subcategory $\SC$ of $\SA$ is called an $n$-cluster tilting subcategory if it satisfies the following conditions.
\begin{itemize}
\item [$(i)$] For every object $A \in \SA$, there exists an admissible monomorphism $A \rightarrowtail C$, which is also a left $\SC$-approximation of $A$.
\item [$(ii)$] For every object $A \in \SA$, there exists an admissible epimorphism $C' \twoheadrightarrow A$, which is also a right $\SC$-approximation of $A$.
\item [$(iii)$] There exists equalities $\SC^{\perp_n}=\SC= {}^{\perp_n}\SC$, where
\[ \SC^{\perp_n} = \{A \in \SA : \Ext_{\SE}^i(C, A)=0 \ {\rm for \ all} \ C \in \SC {\rm \ and \ all} \ 1\leq i \leq n-1\},\]
\[{}^{\perp_n}\SC  = \{A \in \SA : \Ext_{\SE}^i(A, C)=0 \ {\rm for \ all} \ C \in \SC {\rm \ and \ all} \  1\leq i \leq n-1\}.\]
\end{itemize}
\end{definition}

Note that if $n=1$, $\SC=\SA$ is the unique $1$-cluster tilting subcategory of $\SA$. Following theorem provides a source of examples of $n$-exact categories.

\begin{theorem}(\cite[Theorem 4.14]{Ja})\label{Th-n-exact}
Let $\SC$ be an $n$-cluster tilting subcategory of the exact category $(\SA, \SE)$. Set
\[\SX=\{ {\bf C}: C^0 \rightarrowtail C^1 \rt \cdots \rt C^n \twoheadrightarrow C^{n+1} \ | \ {\bf C} \ {\rm is} \ \SE\mbox{-}{\rm acyclic \ and} \ C^i \in \SC, \forall i \in \{0, 1, \ldots, n+1\}\}.\]
Then $(\SC, \SX)$ is an $n$-exact category.
\end{theorem}

\begin{remark}\label{Rem-Construction}
By the above theorem any $n$-cluster tilting subcategory of an exact category is $n$-exact. In particular, $n$-pullbacks and $n$-pushouts exist in $\SC$. Let us explain the construction of $n$-pullback diagrams. $n$-pushout diagrams are constructed similarly, see e.g. \cite[Proposition 2.18]{JK}. Consider the following $n$-pullback diagram
\begin{equation*}
\begin{tikzcd}
& X^0 \rar[tail]{d^0_Y} \dar[equal] & Y^1  \rar{d^1_Y}\dar & Y^2 \rar{d^2_Y} \dar & \cdots \rar &Y^n\rar{d^{n}_Y}\dar & Y^{n+1}  \dar{g}\\
\eta: & X^0 \rar[tail]{d^0_X} & X^1\rar{d_X^1} & X^2 \rar{d^2_X} & \cdots\rar & X^n\rar[two heads]{d^n_X}& X^{n+1}
\end{tikzcd}
\end{equation*}
obtained from $n$-exact sequence $\eta$ along $g$. Each square
\begin{equation*}
\begin{tikzcd}
Y^{i} \ar{dd} \ar{rr} && Y^{i+1}  \ar{dd}{g^{i+1}} \\ \\
X^{i} \ar{rr}{d^i_X} && X^{i+1}
\end{tikzcd}
\end{equation*}
\noindent for $i \in \{1, 2, \ldots, n\}$, is constructed by taking usual pullback diagram of the morphism $({g^{i+1}}, {d^i_X})$ in the ambient exact category $(\SA, \SE)$ and then taking a right $\SC$-approximation of the outcome, that exists by definition of $n$-cluster tilting subcategories. In other words, as it can be seen from the following diagram

\begin{equation*}
\begin{tikzcd}
Y^{i} \ar[dotted]{dd}{g^i} \ar[dotted]{rr}{d^i_Y} \drar{\gamma^i} && Y^{i+1}  \ar{dd}{g^{i+1}} \\
& Z^i \urar{\alpha^{i}} \dlar{\beta^i} \\
X^{i} \ar{rr}{d^i_X} && X^{i+1}
\end{tikzcd}
\end{equation*}

\noindent $(Z^i, \alpha^i, \beta^i)$ is a pullback of $(g^{i+1}, d^i_X)$ and $\gamma^i: Y^i \lrt Z^i$ is a right $\SC$-approximation of $Z^i$.
\end{remark}

\begin{remark}
Let $(\SA, \SE)$ be an exact category, $\SC$ be an $n$-cluster tilting subcategory of $\SA$ and $\SX$ denote the $n$-exact structure of $\SC$. For each $C, C' \in \SC$, set
\[\Yext^n_{\SC}(C, C'):= \{[\delta]  \mid  \delta : C' \rightarrowtail C^1 \rt C^2 \rt \cdots \rt C^n \twoheadrightarrow C \ {\rm is \ an} \ \SE\mbox{-}{\rm acyclic \ complex \ in} \ \SC \}, \]
where $[\delta]$ denotes the Yoneda equivalence class of $\delta$ similar to what is defined in \ref{Yoneda-Ext}. One can follow the same argument as in Subsection 6.2 of \cite{FS} to deduce that $\Yext^n_{\SC}(C, C')$, based on $n$-pullbacks and $n$-pushouts in $\SC$, could be equipped with a Baer sum making it into a group. In fact, $\Yext^n_{\SC}( - , - )$ becomes a bifunctor on $\SC$, see also \cite[Theorem IV.9.1]{HS} for a proof in an abelian category.

Moreover, for a morphism $f: X' \lrt Y'$ in $\SC$, similar argument as in \cite[pp. 197-198]{FS}, implies that the induced morphism
\[f^*:=\Yext^n_{\SC}(X, f): \Yext^n_{\SC}(X, X') \lrt \Yext^n_{\SC}(X, Y')\]
for every $X \in \SC$, can be computed in terms of $n$-pushout diagrams. That is for every element
\[\eta: ~~~~~~~X' \rightarrowtail X^1 \rt X^2 \rt \cdots \rt X^n \twoheadrightarrow X \]
of $\Yext^n_{\SC}(X, X'),$ $f^*(\eta)$ is the $n$-exact sequence obtained by extending $n$-pushout of $\eta$ along $f$, that is
\begin{equation*}
\begin{tikzcd}
\eta:   & X' \rar[tail] \dar{f} & X^1 \rar \dar & \cdots \rar & X^{n-1} \rar \dar & X^n \rar[two heads] \dar & X  \dar[equals]  \\
\eta':  & Y'  \rar[tail] & Y^1 \rar & \cdots \rar & Y^{n-1} \rar & Y^n \rar[two heads] & X
\end{tikzcd}
\end{equation*}
For a similar discussion in the case where $\SA=\mmod\La$ is an abelian category see \cite[Remark 3.8(b) and Lemma 3.13]{Fe}. Similarly the morphism
\[f_*:=\Yext^n_{\SC}(f, X): \Yext^n_{\SC}(Y', X) \lrt \Yext^n_{\SC}(X', X)\]
for every $X \in \SC$ is defined using $n$-pullback diagrams.
\end{remark}

\begin{remark}
Let $R$ be a commutative local ring and $\SA$ be an abelian $R$-category with enough projective objects. Let $\SB$ be a resolving subcategory of $\SA$, that is, $\SB$ contains projective objects of $\SA$ and is closed under extensions and kernel of surjections. Moreover assume that the stable category $\underline{\SB}$ is a dualising $R$-variety \cite{AR1}. Since $\SB$ is a full extension closed subcategory of abelian category $\SA$, it is an exact category. Let $\SE$ denote the exact structure of $\SB$. Now let $\SC$ be an $n$-cluster tilting subcategory of $\SB$. So $\SC$ is an $n$-exact category. Let $\SX$ denote the $n$-exact structure of $\SC$. By Section 2.5.1 of \cite{I2} higher-dimensional Auslander-Reiten theory exists for subcategory $\SC$ of $\SB$. Hence, by \cite[A.1. Proposition]{I1}, if $X^0, X^{n+1} \in\SC$, every element in $\Ext^n_{\SE}(X^{n+1}, X^0)$ is Yoneda equivalent to an admissible $n$-exact sequence
\[\eta: ~~~~~~~X^0\rightarrowtail X^1 \rt \cdots \rt X^n  \twoheadrightarrow X^{n+1}\]
in $\SC$, i.e. all middle terms lie in $\SC$. Therefore, in this case, the bifunctor $\Yext^n_{\SC}$ is the same as the bifunctor $\Ext^n_{\SE}$ restricted to $\SC$. See \cite[Remark 2.9]{Fe} for the special case $\SA=\SB=\mmod\La$, where $\La$ is a finite dimensional artin algebra.
\end{remark}

\begin{notation}
In view of the above remark, under some mild conditions, bifunctor $\Yext^n_{\SC}( - , - )$ is nothing but the usual bifunctor $\Ext^i_{\SE}( - , - )$, recalled in \ref{Yoneda-Ext}, restricted to $\SC$. Because of this fact, from now on, $\Yext^n_{\SC}( - , - )$ will be denoted by $\Ext^n_{\SX}( - , - ).$
\end{notation}

\begin{example}
As an example of the above discussion, set $\SA=\mmod\La$ and let $\SB=\Gprj\La$ be the full subcategory of $\SA$ consisting of all Gorenstein-projective $\La$-modules. By \cite[Theorem 3.16]{AHS}, if $\SC$ is an $n$-cluster tilting subcategory of $\mmod\La$ with enough injectives, then $\SC \cap \Gprj\La$ is an $n$-cluster tilting subcategory with enough injectives of $\Gprj\La$, see also \cite[Theorem 7.3]{Kv}. Hence in this case, $\Ext^n_{\SC \cap \Gprj\La}$ is the same as the functor $\Ext^n_{\Gprj\La}$ on $\SC \cap \Gprj\La$.
\end{example}

Let $(\SA, \SE)$ be an exact category and $f:X\rt A$ and $g:B\rt Y$ be morphisms in $\SA$. Then it is known \cite[Remark 6.7]{FS} that
\[\Ext^n_{\SE}(f, g)=\Ext^n_{\SE}(X, g)\Ext^n_{\SE}(f, B)=\Ext^n_{\SE}(f, Y)\Ext^n_{\SE}(A, g).\]
Here we show that the same result holds true for $\Ext^n_{\SX}$, where $(\SC, \SX)$ is an $n$-cluster tilting subcategory of $(\SA, \SE)$.

\begin{proposition}\label{Ext}
Let $\SC$ be an $n$-cluster tilting subcategory of the exact category $(\SA, \SE)$. Let $f:X\rt A$ and $g:A'\rt Y$ be morphisms in $\SC$. Then
\[\Ext^n_{\SX}(f, g)=\Ext^n_{\SX}(X, g)\Ext^n_{\SX}(f, A')=\Ext^n_{\SX}(f, Y)\Ext^n_{\SX}(A, g),\]
that is, the following diagram of abelian groups is commutative
\[ \xymatrix{  \Ext^n_{\SX}(A, A') \ar[rr]^<<<<<<<<<<{\Ext^n_{\SX}(f, A')}\ar[d]^{\Ext^n_{\SX}(A,g)} && \Ext^n_{\SX}(X,  A')\ar[d]^{\Ext^n_{\SX}(X, g)}\\
\Ext^n_{\SX}(A, Y)  \ar[rr]^<<<<<<<<<{\Ext^n_{\SX}(f, Y)} && \Ext^n_{\SX}(X, Y).   }\]
\end{proposition}

\begin{proof}
Suppose that $\eta:~~~~A' \rightarrowtail X^1 \rt \cdots \rt X^n \twoheadrightarrow A\in\Ext^n_{\SX}(A, A')$ is given. By taking $n$-pullback of $\eta$ along $f$ and $n$-pushout of it along $g$, we get the following diagram.
\[
 \begin{tikzcd}
A'\rar[tail] \dar[equals] & U^1\rar\dar &\cdots\rar & U^n\rar[two heads] \dar & X\dar{f}\\
A' \rar[tail] \dar{g} & X^1\rar\dar &\cdots\rar & X^{n}\rar[two heads] \dar & A\dar[equals]\\
Y\rar[tail] & V^1\rar &\cdots\rar & V^n \rar[two heads] & A
 \end{tikzcd}
 \]
Now by taking $n$-pushout of $A'\rt U^1 \rt \cdots \rt U^n$ along $g$ and applying Proposition 4.9 of \cite{Ja}, we get the following diagram
{\footnotesize{
\begin{equation*}
\begin{tikzcd}[column sep=small]
 &  &  A'\dlar{g}\ar[tail]{rr}\ar[equals]{dd} &&U^1\dlar\ar{rr}\ar{dd} &&\cdots\dlar[dotted]\ar{rr}\ar[dotted]{dd}&&U^n\ar{dd} \dlar\ar[two heads]{rr}&&X\dlar[equals]\ar{dd}{f}\\
 \gamma: &     Y\ar[tail]{rr}\ar[equals]{dd} &&W^1\ar{rr}\ar{dd}&&\cdots\ar{rr}\ar[dotted]{dd}&& W^n\ar[two heads]{rr}\ar{dd}&&X\ar{dd}[near start]{f}&& \\
    &\eta: ~~~~~~~~~&A'\dlar{g}\ar[tail]{rr}&&X^1\ar{rr}\dlar&&\cdots\dlar[dotted]\ar{rr}&&X^n\dlar\ar[two heads]{rr}&&A\dlar[equals]\\
    & Y\ar[tail]{rr}&&V^1\ar{rr}&&\cdots\ar{rr}&&V^n\ar[two heads]{rr}&&A
\end{tikzcd}
\end{equation*}}}

\noindent Note that by (dual of) the statement $(iv) \Rightarrow (i)$ of Proposition 4.8 of \cite{Ja}, diagram
\begin{equation*}
\begin{tikzcd}
  &     Y\ar[tail]{r}\ar[equals]{d} &W^1\ar{r}\ar{d}&\cdots\ar{r}& W^n\ar[two heads]{r}\ar{d}&X\ar{d}{f}& \\
    & Y\ar[tail]{r}&V^1\ar{r}&\cdots\ar{r}&V^n\ar[two heads]{r}&A
\end{tikzcd}
\end{equation*}
is an $n$-pullback diagram. Hence the admissible $n$-exact sequence $\gamma$ can be described as
\[\Ext^n_{\SX}(X, g)\Ext^n_{\SX}(f, A')(\eta)=\gamma = \Ext^n_{\SX}(f, Y)\Ext^n_{\SX}(A,g)(\eta).\]
\end{proof}

We end this section by the definition of $n$-proper classes that will be used later.

\s {\sc $n$-proper classes.}\label{proper class}
Let $(\SC, \SX)$ be an $n$-cluster tilting subcategory of an exact category $(\SA, \SE)$. A class $\SF$ of $\SE$-acyclic complexes of length $n$ in $\SC$ is called an $n$-proper class if it contains all split (contractible) $\SE$-acyclic complexes, is closed under isomorphisms and finite direct sums and is closed under $n$-pullbacks and $n$-pushouts along any other morphisms in $\SA$.

It is easy to see that any $n$-proper class of $\SE$-acyclic complexes of length $n$ gives rise to an additive subfunctor $\Ext^n_{\SF}$ of $\Ext^n_{\SX}$. On the other hand, any additive subfunctor of  $\Ext^n_{\SX}$ induces an $n$-proper class of $\SE$-acyclic complexes of length $n$. For a similar discussion in the classical case $n=1$, see \cite[1.2]{DRSS} for exact categories and \cite{ASo} for abelian categories.

\section{Higher ideal cotorsion pairs}\label{Higher Cotorsion}
In this section we introduce and study a higher version of the notion of ideal cotorsion pairs \cite[Definition 12]{FGHT}. Our setting for this purpose, and throughout the paper, is $n$-cluster tilting subcategories of exact categories. Let us begin by recalling the notion of an ideal of an additive category.

\s{\sc Ideal approximation theory.}
Let $\SA$ be an additive category. A two sided ideal $\SI$ of $\SA$ is a subfunctor
\[ \SI( - , - ): \SA^{\op}\times\SA \lrt \SA b\]
of the bifunctor $\SA ( - , - )$ that associates to every pair $A$ and $A'$ of objects in $\SA$ a subgroup $\SI(A, A')\subseteq \SA(A, A')$ such that
\begin{itemize}
\item[$(i)$] If $f\in \SI(A, A')$ and $g\in \SA(A', C)$, then $gf\in \SI(A, C)$,
\item[$(ii)$] If $f\in \SI(A, A')$ and $g\in \SA(D, A)$, then $fg\in\SI(D, A')$.
\end{itemize}

Let $\SI$ be an ideal of $\SA$ and $A \in \SA$ be an object of $\SA$. An $\SI$-precover of $A$ is a morphism $C \st{\va}{\lrt} A$ in $\SI$ such that any other morphism $C' \st{\va'}{\lrt} A$ in $\SI$ factors through $\va$, i.e. there exists a morphism $\psi: C' \lrt C$ such that $\va\psi=\va'$. $\SI$ is called a precovering ideal if every object $A \in \SA$ admits an $\SI$-precover. The notions of $\SI$-preenvelope and preenveloping ideals are defined dually. See \cite[Section 3]{H1} for the special case in module category and \cite{FGHT} for the general definitions and properties in an exact category.

We say that an object $A$ of $\SA$ is in $\SI$ if the identity morphism $1_A$ is in $\SI(A, A)$.

\begin{notation}
Let $\SM$ be a collection of morphisms in $\SC$. We define the right and left orthogonals of $\SM$, respectively, as follows
\[\SM^{\perp}:= \lbrace g: ~~\Ext^n_{\SX}(m, g)=0~~~~for~~all ~~~m\in\SM\rbrace,\]
\[{}^{\perp}\SM:= \lbrace f: ~~\Ext^n_{\SX}(f, m)=0~~~~for~~all ~~~m\in\SM\rbrace.\]
\end{notation}

Note that maybe the better notations for the above orthogonals are $\SM^{\perp_n}$  and ${}^{\perp_n}\SM$. Since $n$ is fixed throughout, we drop it, both for the ease of notation and also to avoid confusion with the orthogonal notations used in the definition of $n$-cluster tilting subcategories. \\

Following proposition is a higher version of \cite[Proposition 9]{FGHT}. In this theorem and also throughout the section, for the ease of notation, we shall use $\Ext^n$ instead of $\Ext^n_{\SX}$.

\begin{proposition}\label{9}
 Let $\SM$ be a collection of morphisms in $\SC$. Then both $\SM^{\perp}$ and ${}^{\perp}\SM$ are ideals of $\SC$.
\end{proposition}

\begin{proof}
We just prove that $\SM^{\perp}$ is an ideal, the proof of the other one follows similarly. Let $g_1, g_2: A' \rt Y$ be morphisms in $\SM^{\perp}$. We should show that not only $g_1+g_2\in \SM^{\perp}$ but also for all morphisms $h:Y \rt Y'$ and $k: Y'' \rt A'$ in $\SC$, $hg_1, g_1k \in\SM^{\perp}$.
These all follows using the equalities of Proposition \ref{Ext}. To see this, let $f: X \rt A$ be a morphism in $\SM$ and consider the following equalities
\begin{align*}
 \ \ \ \ \ \ \ \ \ \ \ \ \ \ \ \Ext^n(f, g_1+g_2)  =&\ \Ext^n(X, g_1+g_2)\Ext^n(f, A') \\
=& \ [\Ext^n(X,  g_1) + \Ext^n(X, g_2)]\Ext^n(f, A')\\
=&  ~\Ext^n(X, g_1)\Ext^n(f, A')+ \Ext^n(X, g_2)\Ext^n(f, A')\\
=&~ \Ext^n(f, g_1)+ \Ext^n(f, g_2)=0.
\end{align*}
to deduce that $g_1+g_2\in \SM^{\perp}$. To prove $hg_1\in\SM^{\perp}$, by another use of Proposition \ref{Ext} we have the following equalities
\begin{align*}
\ \Ext^n(f, hg_1)  =&\ \Ext^n(X, hg_1)\Ext^n(f, A') \\
=&  ~\Ext^n(X, h)\Ext^n(X, g_1) \Ext^n(f, A')\\
=&~ \Ext^n(X, h)\Ext^n(f, g_1)=0,
\end{align*}
which, in turn, implies that $hg_1\in\SM^{\perp}$.  The statement $g_1k \in\SM^{\perp}$ follows similarly. Hence the proof is complete.
\end{proof}

The above proposition, in particular, implies that $\SI^{\perp}$ and ${}^{\perp}\SI$ are ideals of $\SC$, whenever $\SI$ is so.

\begin{definition}
Let $\SI$ and $\SJ$ be ideals of $\SC$. A pair $(\SI, \SJ)$ is called an $n$-orthogonal pair of ideals if for every $f \in \SI$ and every $g \in \SJ$, the pair $(f, g)$ is an $n$-orthogonal pair of morphisms, that is, the morphism
\[\Ext^n(f, g): \Ext^n(A, B)\lrt \Ext^n(X, Y)\]
of abelian groups is zero.
\end{definition}

For instance, the  pair $(1_A, g)$ is an $n$-orthogonal pair of morphisms if and only if $\Ext^n(A, g)=0$. We infer this using the following equalities
 \[\Ext^n(1_A, g)= \Ext^n(A, g)\Ext^n(1_A, B)= \Ext^n(A, g),\]
proved in Proposition \ref{Ext}. Similar argument implies that the pair $(1_A, 1_B)$ is $n$-orthogonal pair of morphisms if and only if $\Ext^n(A, B)=0$.

\begin{definition}
The $n$-orthogonal pair $(\SI,\SJ)$ of ideals in $\SC$ is called an $n$-ideal cotorsion pair if $\SI = {}^{\perp}{\SJ}$ and $\SJ= \SI^{\perp}$.
\end{definition}

\begin{definition}(see \cite[page 762]{FGHT})\label{Def-Special Precover}
Let $\SI$ be an ideal of $\SC$ and $A \in \SC$ be an arbitrary object. A morphism $i: X^n \rt A$ in $\SI$ is called a special $\SI$-precover of $A$ if it obtained as the rightmost morphism in an $n$-pushout of an admissible $n$-exact sequence $\eta$ along a morphism $j: Y {\lrt} A' \in \SI^{\perp}$. It is depicted by the following diagram
\begin{equation*}
\begin{tikzcd}
\eta: & Y\dar{j}\rar[tail] & Y^1\rar\dar& \cdots \rar& Y^n\rar[two heads]\dar& A\dar[equals]\\
\eta': & A'\rar[tail] & X^1\rar& \cdots\rar& X^n\rar[two heads]{i}& A.
\end{tikzcd}
\end{equation*}
The ideal $\SI$ is called a special precovering ideal if every object $A\in \SC$ has a special $\SI$-precover. The notions of special $\SI$-preenvelopes  and special preenveloping ideals are defined dually.
\end{definition}

\begin{remark}
Let $i: X^n \rt A$ be a special $\SI$-precover of $A$. Then it, necessarily, is an $\SI$-precover of $A$. To see this, let $i': X'\rt A$ be a morphism in $\SI$. We show that it factors through $i$. By the above diagram,
$\eta'= \Ext^n(A, j)(\eta)$.
By assumption $j \in\SI^{\perp}$, so the $n$-pullback of $\eta'$ along $i'$ is
\[\Ext^n(i', A')(\eta')= \Ext^n(i', A')\Ext^n(A, j)(\eta)= \Ext^n(i',j)(\eta)=0,\]
Hence we have the desired factorization
\begin{equation*}
\begin{tikzcd}
& X'\dlar[dotted]\dar{i'}\\
X^n \rar{i}& A.
\end{tikzcd}
\end{equation*}
\end{remark}

\begin{definition}\label{12}
An $n$-ideal cotorsion pair $(\SI,\SJ)$ is called complete if every object in $\SC$ admits a special $\SI$-precover and a special $\SJ$-preenvelope.
\end{definition}

In the following theorem we present a source of examples of $n$-ideal cotorsion pairs.

\begin{theorem}\label{13}
Let $\SI$ be a special precovering ideal of $\SC$. Then the $n$-orthogonal pair of ideals $(\SI, \SI^{\perp})$ is an $n$-ideal cotorsion pair.
\end{theorem}

\begin{proof}
By definition, we just need to show that $\SI={}^{\perp}(\SI^{\perp})$. The inclusion $\SI\subseteq {}^{\perp}(\SI^{\perp})$ is trivial. So assume that $i': X' \rt A$ is in ${}^{\perp}(\SI^{\perp})$. We show that $i'\in\SI$. Since $\SI$ is a special precovering ideal, there is a special $\SI$-precover $i:X^n\rt A$ of $A$, which is the rightmost morphism in an admissible $n$-exact sequence
\[\eta': \ \ A' \rightarrowtail X^1 \rt \cdots\rt X^n \stackrel{i}\twoheadrightarrow A\]
that is obtained from an $n$-pushout diagram of an admissible $n$-exact sequence $\eta$ along a morphism $j: Y\lrt A'$ in $\SI^{\perp}$.
That is $\eta' = \Ext^n(A, j)(\eta)$, for some admissible $n$-exact sequence $\eta$. Now by taking the $n$-pullback of $\eta'$ along $i'$ and using the assumption that $i': X'\lrt A \in {}^{\perp}(\SI^{\perp})$, we have
\[\Ext^n(i', A')(\eta')= \Ext^n(i', A')\Ext^n(A, j)(\eta)=\Ext^n(i', j)(\eta)=0.\] Hence we get the factorization
\begin{equation*}
\begin{tikzcd}
& X'\dlar{g}\dar{i'}\\
X^n \rar{i}& A.
\end{tikzcd}
\end{equation*}
which implies that $i'=ig \in \SI$. The proof is hence complete.
\end{proof}

\section{Higher phantom ideals}\label{Higher Phantom}
In this section, we introduce and study the notion of phantom morphisms in $n$-cluster tilting subcategories of exact categories. It will provide a higher version of the notion of phantom morphisms introduced and studied by Herzog \cite{H1, H2}.

\begin{setup}
Throughout the section, $\SC$ is an $n$-cluster tilting subcategory of an exact category $(\SA, \SE)$ with $n$-exact structure $\SX$.  Moreover, $\SF$ denotes an additive subfunctor of $\Ext^n_{\SX}$. As it is explained in \ref{proper class}, every subfunctor $\SF$ of $\Ext^n_{\SX}$ induces an $n$-proper subclass of $\SX$, that also will be denoted by $\SF$. The admissible $n$-exact sequences in $\SF$ are called $\SF$-admissible $n$-exact sequences. Throughout we denote $\Ext^n_{\SX}$ by $\Ext^n$, for simplicity.
\end{setup}

Let us begin with the following definition.

\begin{definition}
With the above notations, a morphism $\varphi$ in $\SC$ is called an $n$-$\SF$-phantom morphism if the $n$-pullback of every $\SX$-admissible $n$-exact sequence along $\varphi$ is an $\SF$-admissible $n$-exact sequence. In other words, $\varphi: X \lrt A$ in $\SC$ is an $n$-$\SF$-phantom morphism if for every object $A'$ in $\SC$, the morphism
\[\Ext^n(\varphi, A'): \Ext^n(A, A') \lrt \Ext^n(X, A')\]
of abelian groups takes values in the subgroup $\SF(X, A')$. We denote the collection of  all $n$-$\SF$-phantom morphisms by $\Phi(\SF)$. Note that it is easy to see that $\Phi(\SF)$ forms an ideal of $\SC$.
\end{definition}

To see an example of phantom morphisms we need to recall following definitions/notations from \cite[page 759]{FGHT}.

\begin{definition}
A morphism $f: X \lrt A$ in $\SC$ is called $\SF$-projective if  for every object $B$ in $\SC$, $\SF(f, B)=0$. In other words,  $f:X\rt A$ in $\SC$ is $\SF$-projective if the $n$-pullback of any $\SF$-admissible $n$-exact sequence along $f$ is contractible. An object $A$ in $\SC$ is called $\SF$-projective if the identity morphism is an $\SF$-projective morphism. The ideal of $\SF$-projective morphisms is denoted by $\SF\mbox{-}{\rm proj}$. The notions of $\SF$-injective morphisms and $\SF$-injective objects are defined dually. The ideal of $\SF$-injective morphisms is denoted by $\SF\mbox{-}{\rm inj}$.
\end{definition}

\begin{example}{\sc ($n$-pure phantom morphisms)}
In this example, we assume that $(\SA, \SE)$ is an exact category with arbitrary direct sums. An object $C$ of $\SA$ is called a compact object \cite[Definition 2.3]{EN} if any morphism from $C$ to a nonempty coproduct $\bigoplus_{i \in I}A_i$ factors through a sub-coproduct $\bigoplus_{j \in J}A_j$, where $J \subseteq I$ is a finite subset.
$\SA$ is called compactly generated if for every $A \in \SA$ there is an admissible epimorphism $h: \bigoplus_{i \in I}C_i \lrt A$ such that $C_i$ is compact for each $i \in I$.  

Let $(\SA, \SE)$ be a compactly generated exact category and $(\SC, \SX)$ be an $n$-cluster tilting subcategory of $\SA$ containing all compact objects. An admissible $n$-exact sequence
\[A' \rightarrowtail X^1 \rightarrow \cdots \rightarrow X^n \twoheadrightarrow A \]
in $\SC$ is called pure $n$-exact if for every compact object $C$, the induced sequence
\[0 \lrt \SC(C, A') \lrt \SC(C, X^1) \lrt \SC(C, X^2) \lrt \cdots \lrt \SC(C, X^n) \lrt \SC(C, A) \lrt 0\]
of abelian groups is acyclic \cite[Definition 2.4]{EN}.

We claim that the collection of all pure $n$-exact sequences forms an $n$-proper class of $\SX$. To see this, note that it is obviously closed under isomorphisms and also contains all contractible $n$-exact sequences. It also is closed under $n$-pushout by any other morphism. We show that it is closed under $n$-pullback along any other morphism. Let
\[\eta: A' \rightarrowtail X^1 \rightarrow \cdots \rightarrow X^n \twoheadrightarrow A \]
be a pure $n$-exact sequence and consider its $n$-pullback along a morphism $A'' \lrt A$

\begin{equation*}
\begin{tikzcd}
\delta: & A' \dar[equals] \rar[tail] & Y^1 \rar\dar{g}& Y^2\rar\dar &\cdots\rar & Y^n\rar[two heads]{d^n_Y} \dar & A'' \dar{h}\\
\eta: & A' \rar[tail] & X^1\rar & X^2\rar &\cdots\rar & X^n \rar[two heads]{d^n_X} & A
\end{tikzcd}
\end{equation*}

Now let $u: C \lrt A''$ be a morphism from a compact object $C$ to $A''$. Since $\eta$ is pure $n$-exact, $hu$ factors through $d^n_X$. Now the construction of the $n$-pullback diagram explained in Remark \ref{Rem-Construction}, implies easily that $u$ factors also through $d^n_Y$. Hence the claim follows.

So, by \ref{proper class}, it gives rise to a subfunctor of $\Ext^n$, that we will denote by $n$-$\rm{Pext}$. By an $n$-pure phantom morphism we mean an $n$-$\rm{Pext}$-phantom morphism. The collection of all $n$-pure phantom morphisms will be denoted by $\Phi(n\mbox{-}\rm{Pext})$. The $n$-$\rm{Pext}$-injective objects will be called $n$-pure injective objects.  The pair of ideals $(\Phi(n\mbox{-}\rm{Pext})$, $n$-$\rm{Pext}$-${\rm inj})$ is an $n$-orthogonal pair.  To see this, let $f$ be an $n$-pure phantom and $g$ be an $n$-pure injective morphism. Then  $n$-pullback of admissible $n$-exact sequence $\eta$ along $f$ is a pure admissible $n$-exact sequence and $n$-pushout of this sequence along $g$ is contractible. Hence $\Ext^n(f, g)=0$.
\end{example}

\begin{notation}
Let $\SI$ be an ideal of $\SC$. The collection of all admissible $n$-exact sequences that are obtained from $n$-pullbacks of admissible $n$-exact sequences along morphisms in $\SI$ is denoted by $\PB(\SI)$.
\end{notation}

\begin{proposition}\label{7}
Let $\SI$ be an ideal of $\SC$. Then $\PB(\SI)$ is an $n$-proper subclass of $\SX$.
\end{proposition}

\begin{proof}
By \ref{proper class}, we have to show that $\PB(\SI)$ contains contractible $n$-exact sequences, is closed under isomorphisms, closed under direct sums, $n$-pullbacks and $n$-pushouts. Obviously it is closed under isomorphisms. Let $\eta: \ A' \rt X^1\rt \cdots\rt X^n\rt A$ be  a contractible $n$-exact sequence. Then $n$-pullback of $\eta$ along morphism $0: A\rt A$ is just $\eta$, so $\eta\in \PB(\SI)(A, A')$. Now let $\eta$ and $\eta'$ be in $\PB(\SI)$ such that arise as $n$-pullbacks of admissible $n$-exact sequences $\gamma$ and $\gamma'$ along morphisms $i$ and $i'$, respectively. Then direct sum of $\eta$ and $\eta'$ is  the $n$-pullback of  $\gamma\oplus\gamma'$ along morphism
$\begin{bmatrix}
 i & 0\\  0 & i'
\end{bmatrix}$,
and so $\PB(\SI)$ is closed under direct sums. To see  that $\PB(\SI)$ is closed under $n$-pullbacks, let $\eta: \ A' \rt X^1\rt \cdots\rt X^n\rt A$ be an admissible $n$-exact sequence that is obtained from the $n$-pullback of $\gamma$ along morphism $i: A\rt A''$. The $n$-pullback of  $\eta$ along morphism $f: X\rt A$ is an $n$-pullback of $\gamma$ along $if$ which is in $\SI$. Finally we show that $\PB(\SI)$ is closed under $n$-pushouts. Let $\eta: \ A' \rt X^1\rt \cdots\rt X^n\rt A$ be the $n$-pullback of $\gamma: \ A' \rt Y^1\rt \cdots \rt Y^n\rt A''$ along $i:A\rt A''$. By Proposition \ref{Ext}, the $n$-pushout of $\eta$ along $g:A' \rt Z$ is
\[\Ext^n_{\SX}(A, g)(\eta)=\Ext^n_{\SX}( A, g)\Ext^n_{\SX}(i, B)(\gamma)=\Ext^n_{\SX}(i, Z)\Ext^n_{\SX}(A'', g)(\gamma),\]
which is the $n$-pullback of admissible $n$-exact sequence $\Ext^n_{\SX}(A'', g)(\gamma)$ along $i$.
\end{proof}

\begin{remark}\label{PB(I)}
The above proposition, in view of \ref{proper class}, implies that $\PB(\SI)$ induces an additive subfunctor of $\Ext^n_{\SX}$, we denote it by the same notation $\PB(\SI)$.
\end{remark}

\begin{lemma}\label{10}
 Let $\SI$ be an ideal of $\SC$.  Then $\SI^{\perp}=\SF\mbox{-}{\rm inj}$, where $\SF=\PB(\SI)$.
\end{lemma}

\begin{proof}
This follows from the fact that a morphism $j$ is $\PB(\SI)$-injective if and only if for every $i\in\SI$, $\Ext^n(i, j)=0$ and this holds true if and only if $j \in \SI^{\perp}$.
\end{proof}

\begin{proposition}\label{14}
Let $\SI$ be a special precovering ideal. Then $\SI$ is the ideal of $\PB(\SI)$-phantom morphisms, that is $\SI=\Phi(\PB(\SI))$.
\end{proposition}

\begin{proof}
It is clear that every morphism in $\SI$ is an $n$-$\PB(\SI)$-phantom morphism. Now  suppose that $i$ is an $n$-$\rm{PB}(\SI)$-phantom morphism, so every $n$-pullback of an admissible $n$-exact sequence along $i$ is a $\PB(\SI)$-admissible $n$-exact sequence.  If $j\in\SI^{\perp}$, then by Lemma \ref{10}, $j$ is $\PB(\SI)$-injective, so $\Ext^n(i, j)=0$ and $i\in {}^{\perp}(\SI^\perp)$.  Since by Theorem \ref{13}, $\SI={}^{\perp}(\SI^\perp)$, we conclude that $i \in \SI$.
\end{proof}

\begin{corollary}
Let $\SI$ be a special precovering ideal of $\SC$. Set $\SF=\PB(\SI)$. Then $(\Phi(\SF), \SF\mbox{-}{\rm inj})$ is an $n$-ideal cotorsion pair.
\end{corollary}

\begin{proof}
By Proposition \ref{14}, $\SI=\Phi(\SF)$. Also by Lemma \ref{10}, $\SI^{\perp}=\SF\mbox{-}{\rm inj}$. Now the result follows by Theorem \ref{13}.
\end{proof}

Following lemma is a special case of the Obscure axiom \cite[ Proposition 4.11]{Ja}.

\begin{lemma}\label{5}
Let $i_0: A' \lrt Y^1$ be an $\SF$-admissible monomorphism that factors through an $\SX$-admissible monomorphism $i: A' \lrt X^1$
\begin{equation*}
\begin{tikzcd}
A' \rar{i}\drar{i_0}& X^1 \dar{g}\\
&  Y^1
\end{tikzcd}
\end{equation*}
Then $i$ is an $\SF$-admissible monomorphism.
\end{lemma}

\begin{proof}
Since $i$ is an $\SX$-admissible monomorphism, there exists an $\SX$-admissible $n$-exact sequence
\[ A' \st{i}{\rightarrowtail} X^1\lrt X^2 \lrt \cdots\lrt X^n \twoheadrightarrow A.\]
Similarly,  for $i_0$ we have $\SF$-admissible $n$-exact sequence
\[ A' \st{i_0}{\rightarrowtail} Y^1 \lrt Y^2 \lrt \cdots\lrt Y^n \twoheadrightarrow A.\]
Using the morphism $g: X^1 \rt Y^1$ and factorization property of the weak cokernels we can construct the commutative diagram
\begin{equation*}
\begin{tikzcd}
\eta: & A' \dar[equals] \rar[tail]{i} & X^1 \rar\dar{g}& X^2\rar\dar[dotted] &\cdots\rar & X^n\rar[two heads] \dar[dotted] & A\dar[dotted]{h}\\
\delta: & A' \rar[tail]{i_0} & Y^1\rar & Y^2\rar &\cdots\rar & Y^n \rar[two heads] & A''
\end{tikzcd}
\end{equation*}
By the dual of $(iv) \Rightarrow (i)$ of Proposition 4.8 of \cite{Ja}, this diagram is an $n$-pullback diagram. Now since $\delta$ is an $\SF$-admissible $n$-exact sequence and $\SF$ is an $n$-proper class, $\eta$ also should be $\SF$-admissible $n$-exact sequence and so $i$ is an $\SF$-admissible monomorphism.
\end{proof}

We say that an additive subfunctor $\SF\subseteq \Ext^n$ has enough injective morphisms if for every object $A' \in\SC$ there exists an $\SF$-admissible $n$-exact sequence
\[\eta: \ \ A' \stackrel{e}\rt X^1 \rt \cdots \rt X^n\rt A\]
where $e: A' \rt X^1$ is an $\SF$-injective morphism. The subfunctor $\SF\subseteq\Ext^n$ has enough special injective morphism if for every object $A' \in\SC$, there exists an $\SF$-admissible sequence $\eta$ as above that obtains from the $n$-pullback of an admissible $n$-exact sequence along an $\SF$-phantom morphism.

\begin{proposition}\label{15}
Let $\SF\subseteq \Ext^n$ be an additive subfunctor admitting enough injective morphisms. Then $\Phi(\SF)= {}^{\perp}(\SF\mbox{-}{\rm inj})$.
\end{proposition}

\begin{proof}
First note that the pair of ideals $(\Phi(\SF), \SF\mbox{-}{\rm inj})$  is an $n$-orthogonal pair. To see this, let $f$ be an $n$-$\SF$-phantom and $g$ be an $\SF$-injective morphism. Then  $n$-pullback of admissible $n$-exact sequence $\eta$ along $f$ is an $\SF$-admissible $n$-exact sequence and $n$-pushout of this sequence along $g$ is trivial. Hence $\Ext^n(f, g)=0$. This, in particular, implies that $\Phi(\SF)\subseteq {}^{\perp}(\SF\mbox{-}\rm inj)$. Now we show that the converse inclusion ${}^{\perp}(\SF\mbox{-}\rm inj)\subseteq\Phi(\SF)$ also holds true. Assume that $f:X \rt A \in {}^{\perp}(\SF\mbox{-}\rm inj)$. We show that $f$ is an $n$-$\SF$-phantom morphism, that is, the $n$-pullback of any admissible $n$-exact sequence
\[\eta: \ \ A' \rightarrowtail X^1 \rt \cdots \rt X^n \twoheadrightarrow A\]
along $f$ is an $\SF$-admissible sequence.

Since $\SF$ has enough injective morphisms, there exists $\SF$-injective $\SF$-admissible monomorphism $e: A' \rt Y$.  By Proposition 4.9 of \cite{Ja}, we get the following diagram
{\footnotesize{
\begin{equation*}
\begin{tikzcd}[column sep=small]
& \eta':&  A'\dlar{e}\ar[tail]{rr}{i}\ar[equals]{dd} &&U^1\dlar\ar{rr}\ar{dd} &&\cdots\dlar[dotted]\ar{rr}\ar[dotted]{dd}&&U^n\ar{dd} \dlar
\ar[two heads]{rr}&&X\dlar[equals]\ar{dd}[near start]{f}\\
\gamma: &     Y\ar[tail]{rr}\ar[equals]{dd} &&W^1\ar{rr}\ar{dd}&&\cdots\ar{rr}\ar[dotted]{dd}&& W^n\ar[two heads]{rr}\ar{dd}&&X\ar{dd}[near start]{f}&& \\
&\eta: \ \ \ \ \ &A'\dlar{e}\ar[tail]{rr}&&X^1\ar{rr}\dlar&&\cdots\dlar[dotted]\ar{rr}&&X^n\dlar\ar[two heads]{rr}&&A\dlar[equals]\\
& Y\ar[tail]{rr}&&V^1\ar{rr}&&\cdots\ar{rr}&&V^n\ar[two heads]{rr}&&A
\end{tikzcd}
\end{equation*}}}

\noindent The admissible sequence $\gamma$ is contractible, since $\Ext^n(f, e)=0$. So the  $\SF$-admissible monomorphism $e$ factors as $e= gi$ for some $g: U^1\rt Y$. By Lemma \ref{5}, the  morphism $i$ is an $\SF$-admissible monomorphism. Thus $\eta'$, $n$-pullback  of  admissible $n$-exact sequence $\eta$ along $f$, is an $\SF$-admissible sequence.
\end{proof}

\begin{corollary}\label{implication 3 to 4}
Let $\SI$ be  an ideal of $\SC$ such that $n$-ideal cotorsion pair $(\SI, \SI^{\perp})$  is complete. Set $\SF=\PB(\SI)$. Then $\SF \subseteq \Ext^n$ is an additive subfunctor admitting enough special injective morphisms. %In particular, $(\Phi(\SF), \SF\mbox{-}{\rm inj})$ is an $n$-ideal cotorsion pair.
\end{corollary}

\begin{proof}
By \ref{proper class} we know that $\SF$ is an additive subfunctor of $\Ext^n$. We show that it has enough special injective morphisms. By Lemma \ref{10} we know that $\SI^{\perp}=\SF\mbox{-}\rm {inj}$. Since the $n$-ideal cotorsion pair $(\SI, \SI^{\perp})$ is complete, $\SF\mbox{-}\rm {inj}$ is special preenveloping ideal. So every $A'\in \SC$ has a special
$\SF\mbox{-}\rm {inj}$-preenvelope $A'\st{j}\rt X^1$ such that  it obtained as the leftmost morphism in an $n$-pullback of an admissible $n$-exact sequence $\eta$ along a morphism $i: A {\lrt} Y \in {}^\perp (\SF\mbox{-}\rm {inj})$. It is depicted by the following diagram.
\begin{equation*}
\begin{tikzcd}
\eta': & A'\rar[tail]{j}\dar[equals] & X^1\rar\dar &\cdots\rar & X^n\dar\rar[two heads] & A\dar{i}\\
\eta: & A'\rar[tail] & Z^1\rar& \cdots \rar& Z^n\rar[two heads]& Y\\
\end{tikzcd}
\end{equation*}
 Also since the $n$-ideal cotorsion pair $(\SI, \SI^{\perp})$ is complete, $\SI$ is special precovering. So by Proposition \ref{14} we get $\SI=\Phi(\SF)$. On the other hand $\SI={}^{\perp}(\SF\mbox{-}\rm {inj})$ implies that $i\in\Phi(\SF)$. Hence we get admissible $n$-exact sequence $\eta'$ with leftmost morphism $j\in\SF\mbox{-}\rm {inj}$, obtained from the  $n$-pullback of an admissible $n$-exact sequence $\eta$ along an $\SF$-phantom morphism. Therefore  $\SF$ has enough special injective morphisms.
\end{proof}

\section{Salce's Lemma}\label{Higher Salce}
Salce's Lemma \cite{S} is one of the main theorems in the classical approximation theory. It relates the notions of (special) precoverings, (special) preenvelopings and cotorsion pairs. By introducing an interesting exact structure on the morphism category of an exact category, called ME-exact structure, an ideal version of Salce's Lemma is proved in \cite[Theorem 6.3]{FH}. Our aim in this section is to provide a higher ideal version of this result.

\begin{setup}
Throughout the section, $n \geq 1$ is a fixed integer and $\SC$ is an $n$-cluster tilting subcategory of the exact category $(\SA, \SE)$ with the $n$-exact structure $\SX$. Since $(\SC, \SX)$ is fixed, for simplicity, we just write $\Ext^n$ instead of $\Ext^n_{\SX}$.
\end{setup}

The notions of injective and projective objects in an $n$-exact category are defined by using the exactness of the Hom functor. For example, an object $E \in \SA$ is called $\SX$-injective if, for every admissible monomorphism $f: X^0 \rt X^1$ in $\SX$, the induced morphism $\SC(X^1, E) \lrt \SC(X^0, E)$ is an epimorphism.

\begin{definition}(\cite[Definition 5.3]{Ja})
We say that an $n$-exact category $(\SC, \SX)$ has enough $\SX$-injectives if for every object $X\in \SC$ there exists injective objects $E^1, \cdots,  E^n$ and an admissible $n$-exact sequence
\[X\rightarrowtail E^1\rightarrow\cdots\rightarrow E^n\twoheadrightarrow X'.\]
By abuse of notation, we denote $X'$ by $\Omega^{-n}X$ and will call it the $n$th $\SX$-cosyzygy of $X$. The notion of having enough $\SX$-projectives is defined dually.
The $n$th $\SX$-syzygy of $X$ is denoted by $\Omega^nX$.
\end{definition}

\begin{definition}
Let $(\SC, \SX)$ be an $n$-cluster tilting subcategory of $\SA$ and
\begin{equation*}
\begin{tikzcd}
X^0 \dar{f^0} \rar[tail] & X^1 \rar \dar{f^1} & X^2 \rar\dar{f^2} & \cdots \rar & X^n \rar[two heads] \dar{f^n} &  X^{n+1} \dar{f^{n+1}}\\
Y^0 \rar[tail] & Y^1 \rar & Y^2 \rar & \cdots \rar & Y^n \rar[two heads]  & Y^{n+1}
\end{tikzcd}
\end{equation*}
be a morphism of admissible $n$-exact sequences in $\SX$. Let $\SI$ be an ideal of $\SC$. We say that $\SI$ is closed under $n$-extensions by $\SX$-injective objects, if whenever $f^0 \in \SI$ and $X^{n+1}$ is an $\SX$-injective object, then we can deduce that all the middle morphisms $f^i$, for $i\in \{1, 2, \ldots, n\},$ are in $\SI$. Dually one can define the notion of an ideal closed under $n$-coextensions by $\SX$-projective objects.
\end{definition}
Note that if $n=1$, and hence $\SC=\SA$, the ideal $\SM^{\perp}$, where $\SM$ is a collection of morphisms in $\SA$, is always closed under ($1$-)extensions by injective objects
\cite[Proposition 9]{FGHT}.

\begin{construction}\label{Rem-Complete}
Consider the following commutative diagram
{\footnotesize{
\begin{equation*}
\begin{tikzcd}[column sep=small]
&&&&&&&&&&&&  U^{n+1} \dlar \ar{dd}[near start]{\varphi} \\
&&&&&&&&&&& V^{n+1} \ar{dd}[near start]{\psi}\\
& \eta: & X^0 \dlar[near start]{f^0} \ar[tail]{rr} && X^1  \dlar \ar{rr} && X^2 \ar{rr} \dlar && \cdots \ar{rr}\dlar[dotted] && X^n \dlar \ar[two heads]{rr} && X^{n+1} \dlar \\
\delta: & Y^0 \ar[tail]{rr} && Y^1 \ar{rr} && Y^2 \ar{rr} && \cdots \ar{rr} && Y^n \ar[two heads]{rr} && Y^{n+1}
\end{tikzcd} \end{equation*}}}

\noindent in which rows are admissible $n$-exact sequences. By taking $n$-pullback diagrams of $\eta$ along $\varphi$ and of $\delta$ along $\psi$ and then using the properties of $n$-pullbacks in $n$-exact categories, the above diagram can be completed as follows

{\footnotesize{
\begin{equation*}
\begin{tikzcd}[column sep=small]
&&X^0 \dlar{f^0} \ar[tail]{rr}\ar[equals]{dd} && U^1\dlar\ar{rr}\ar{dd} && U^2 \ar{rr}\ar{dd} \dlar && \cdots \ar{rr}\ar[dotted]{dd}\dlar[dotted] && U^n \dlar \ar[two heads]{rr} \ar{dd} && U^{n+1} \dlar \ar{dd}[near start]{\varphi} \\
& Y^0 \ar[tail]{rr} \ar[equals]{dd} && V^1\ar{rr}\ar{dd}&& V^2\ar{rr}\ar{dd} && \cdots \ar{rr}\ar[dotted]{dd} && V^n \ar[two heads]{rr}\ar{dd} && V^{n+1} \ar{dd}[near start]{\psi} \\
&& X^0 \dlar{f^0} \ar[tail]{rr} && X^1  \dlar \ar{rr} && X^2 \ar{rr} \dlar && \cdots \ar{rr}\dlar[dotted] && X^n \dlar \ar[two heads]{rr} && X^{n+1} \dlar \\
& Y^0 \ar[tail]{rr} && Y^1 \ar{rr} && Y^2 \ar{rr} && \cdots \ar{rr} && Y^n \ar[two heads]{rr} && Y^{n+1}
\end{tikzcd}
\end{equation*}}}

\noindent where all squares are commutative and rows are admissible $n$-exact sequences. To construct morphisms $U^i \lrt V^i$, one should apply the construction of $n$-pullbacks, explained in Remark \ref{Rem-Construction}.
\end{construction}

Now we are in a position to state and prove the main result of this section. For a proof in the case $n=1$, see \cite[Theorem 6.3]{FH}.

\begin{theorem}(Salce's Lemma)\label{18}
Let $(\SC, \SX)$ be an $n$-cluster tilting subcategory of an exact category $(\SA, \SE)$. Let $(\SI, \SJ)$ be an $n$-ideal cotorsion pair such that $\SI$ is closed under $n$-coextensions by $\SX$-projective objects and $\SJ$ is closed under $n$-extensions by $\SX$-injective objects. If $\SC$ has enough $\SX$-injective objects, then $\SI$ is a special precovering ideal if and only if $\SJ$ is a special preenveloping ideal.
\end{theorem}

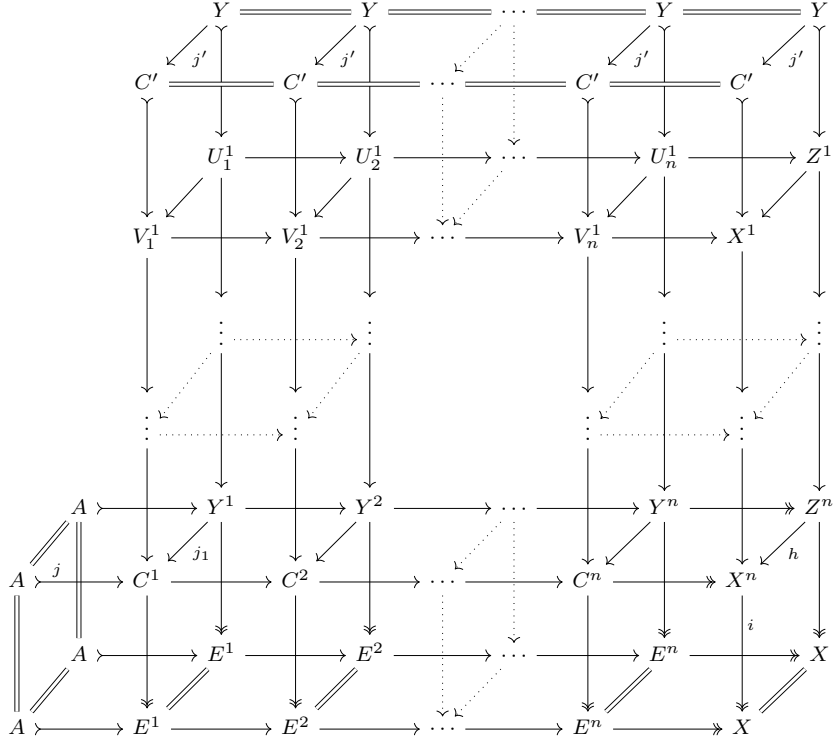
\begin{proof}
We assume that $\SI$ is a special precovering ideal and prove that $\SJ$ is special preenveloping. The converse statement can be proved dually.
Let $A$ be an object of $\SC$. Since $\SC$ has enough $\SX$-injective objects, there is an admissible $n$-exact sequence
$\eta: A \rightarrowtail E^1 \rt \cdots \rt E^n \twoheadrightarrow X$ such that $E^i$, for $1 \leq i \leq n$, is $\SX$-injective. Since $\SI$ is special precovering, there exists an $n$-pushout diagram
\begin{equation*}
\begin{tikzcd}
Y \dar{j'} \rar[tail] & Z^1 \rar \dar & \cdots \rar & Z^{n-1} \rar \dar & Z^{n} \rar[two heads] \dar{h} &  X \dar[equal]\\
C' \rar[tail] & X^1 \rar & \cdots \rar & X^{n-1} \rar & X^n \rar[two heads]{i}  & X
\end{tikzcd}
\end{equation*}
where $i: X^n \rt X$ in $\SI$ is a special $\SI$-precover of $X$ and $j' \in \SI^{\perp}=\SJ$. Now consider $n$-pullback of $\eta$ along $i: X^n \rt X$ to get the following diagram
\begin{equation*}
\begin{tikzcd}
A \dar[equal] \rar[tail]{j} & C^1 \rar \dar & C^2 \rar \dar & \cdots \rar & C^{n} \rar[two heads] \dar &  X^n \dar{i} \\
A \rar[tail] & E^1 \rar & E^2 \rar & \cdots \rar & E^n \rar[two heads]  & X
\end{tikzcd}
\end{equation*}
To complete the proof it is enough to show that $j: A \rt C^1$ belongs to $\SJ$.  To this end, take $n$-pullback of $\eta$ along $ih$, to get the following commutative diagram
{\footnotesize{
\begin{equation*}
\begin{tikzcd}[column sep=small]
&&&&&&&&&&&&Y\dlar{j'}\ar[tail]{dd}\\
&&&&&&&&&&&C'\ar[tail]{dd}\\
&&&&&&&&&&&&Z^1\ar{dd}\dlar\\
&&&&&&&&& &&X^1\ar{dd}\\
&&&&&&&&&&&&\vdots\ar{dd}\dlar[dotted]\\
&&&&&&&&&&&\vdots\ar{dd}\\
%&&&&&&&&&& &&Z^{n-1}\dlar\ar{dd}\\
%&&&&&&&&&&& X^{n-1}\ar{dd}  \\
&&A\dlar[equals]\ar[tail]{rr}\ar[equals]{dd}&&Y^1\dlar\ar{rr}\ar{dd}&&Y^2\ar{rr}\ar{dd}\dlar&&\cdots\ar{rr}\dlar[dotted]\ar[dotted]{dd}&&Y^n\dlar\ar[two heads]{rr}\ar{dd}&& Z^n\dlar{h}\ar[two heads]{dd} \\
&A\ar[tail]{rr}[near start]{j}\ar[equals]{dd} && C^1\ar{rr}\ar{dd}&&C^2\ar{rr}\ar{dd}&& \cdots \ar{rr}\ar[dotted]{dd} && C^n \ar[two heads]{rr}\ar{dd} && X^n\ar[two heads]{dd}[near start]{i}\\
&&A\dlar[equals]\ar[tail]{rr}&&E^1\ar{rr}\dlar[equals]&&E^2\ar{rr}\dlar[equals]&&\cdots\ar{rr}\dlar[dotted]&&E^n\dlar[equals]\ar[two heads]{rr}&&X\dlar[equals]\\
&A\ar[tail]{rr}&&E^1\ar{rr}&& E^2\ar{rr}&&\cdots\ar{rr}&&E^n\ar[two heads]{rr}&&X
\end{tikzcd}
\end{equation*}}}

Now we start from the rightmost column of the above diagram and use Construction \ref{Rem-Complete} step by step to get the following diagram, in which all rows and columns are admissible $n$-exact sequences.
{\footnotesize{
\begin{equation*}
\begin{tikzcd}[column sep=small]
&&&&Y\ar[equals]{rr}\ar[tail]{dd}\dlar{j'}&&Y\dlar{j'}\ar[equals]{rr}\ar[tail]{dd}&&\cdots\ar[equals]{rr}\dlar[dotted] \ar[dotted]{dd} &&Y\ar[equals]{rr}\ar[tail]{dd}\dlar{j'}&&Y\dlar{j'}\ar[tail]{dd}\\
&&&C'\ar[equals]{rr}\ar[tail]{dd}&&C'\ar[equals]{rr}\ar[tail]{dd}&&\cdots\ar[equals]{rr}\ar[dotted]{dd} &&C'\ar[equals]{rr}\ar[tail]{dd}&&C'\ar[tail]{dd}\\
&&&&U_1^1\ar{dd}\ar{rr}\dlar&&U_2^1\dlar\ar{rr}\ar{dd}&&\cdots\ar{rr}\dlar[dotted] &&U_n^{1}\ar{dd}\ar{rr}\dlar&&Z^1\ar{dd}\dlar\\
&&&V_1^1\ar{rr}\ar{dd}&&V_2^1\ar{rr}\ar{dd}&&\cdots\ar{rr}&&V_n^1\ar{rr}\ar{dd}&& X^1\ar{dd}\\
&&&&\vdots \ar{dd}\dlar[dotted]\ar[dotted]{rr} &&\vdots \ar{dd}\dlar[dotted]&&&&\vdots\ar{dd}\dlar[dotted]\ar[dotted]{rr} &&\vdots\ar{dd}\dlar[dotted]\\
&&&\vdots\ar{dd}\ar[dotted]{rr}&&\vdots \ar{dd}&&&&\vdots\ar{dd}\ar[dotted]{rr}&&\vdots\ar{dd}\\
%&&&&U_1^{n-1}\ar{rr}\ar{dd}\dlar&&U_2^{n-1}\ar{rr}\ar{dd}\dlar&&\cdots\ar{rr}\ar{dd}\dlar[dotted]&&U_n^{n-1}\ar{dd}\ar{rr}\dlar&& Z^{n-1}\dlar\ar{dd}\\
%&&&V_1^{n-1}\ar{dd}\ar{rr}&&V_2^{n-1}\ar{rr}\ar{dd}&&\cdots\ar{rr}\ar{dd}&&V_n^{n-1}\ar{rr}\ar{dd}&& X^{n-1}\ar{dd}  \\
&&A\dlar[equals]\ar[tail]{rr}\ar[equals]{dd}&&Y^1\dlar{j_1}\ar{rr}\ar[two heads]{dd} && Y^2\ar{rr}\dlar\ar[two heads]{dd} &&\cdots\ar{rr}\dlar[dotted]\ar[dotted]{dd} && Y^n\dlar\ar[two heads]{rr}\ar[two heads]{dd}&& Z^n\dlar{h}\ar[two heads]{dd} \\
& A\ar[tail]{rr}[near start]{j}\ar[equals]{dd} && C^1\ar{rr}\ar[two heads]{dd}&& C^2\ar{rr}\ar[two heads]{dd}&&\cdots \ar{rr}\ar[dotted]{dd}  && C^n
\ar[two heads]{rr} \ar[two heads]{dd} && X^n\ar[two heads]{dd}[near start]{i}\\
&&A\dlar[equals]\ar[tail]{rr}&&E^1\ar{rr}\dlar[equals]&&E^2\ar{rr}\dlar[equals]&&\cdots \ar{rr}\dlar[dotted] &&E^n\dlar[equals]\ar[two heads]{rr}&&X\dlar[equals]\\
&A\ar[tail]{rr}&&E^1\ar{rr}&& E^2\ar{rr}&&\cdots\ar{rr}&&E^n\ar[two heads]{rr}&&X
\end{tikzcd}
\end{equation*}}}
 Since $\SJ$ is closed under $n$-extensions by injective objects, we infer that $j_1$ is in $\SJ$. Hence $j$ factors through $j_1$ and so $j \in \SJ$, as we desired.
\end{proof}

\section{Special precovering ideals}\label{Special Precovering}
The main purpose of this section is to study connections between special precovering ideals and $n$-phantom morphisms in an $n$-cluster tilting subcategory. In particular, we show that every special precovering ideal, under some conditions, can be represented as an ideal of $n$-$\SF$-phantom morphisms, for some bifunctor $\SF$ of $\Ext^n$. Throughout the section $(\SC, \SX)$ is an $n$-cluster tilting subcategory of an exact category $(\SA, \SE)$.\\

Let us begin with the following useful lemma.

\begin{lemma}\label{16}
Let $A\in\SC$ and $\gamma:~~ K\rightarrowtail Y^1\rt  \cdots \rt Y^n \st{p}\twoheadrightarrow A$ be an admissible $n$-exact sequence such that $p$ is a projective morphism. Then a morphism $\varphi:  X \rt A$ is an $n$-$\SF$-phantom morphism if and only if the $n$-pullback of $\gamma$ along $\varphi$ is an $\SF$-admissible $n$-exact sequence.
\end{lemma}

\begin{proof}
The `only if' part is obvious, by definition. For the proof of the `if' part,
Let \[\eta: \ \ A' \rightarrowtail X^1 \rt \cdots\rt X^n \twoheadrightarrow  A\] be an arbitrary admissible $n$-exact sequence. We show that $n$-pullback of $\eta$ along $\varphi$ is an $\SF$-admissible $n$-exact sequence. Since $p$ is a projective morphism, there is a morphism $Y^n\rt X^n$ that makes the rightmost square of the following diagram commutative. Now by the factorization property of weak kernels we can complete the diagram and in fact get a morphisms of admissible $n$-exact sequences
\begin{equation*}
\begin{tikzcd}
\gamma :& K\dar{g}\rar[tail] &  \cdots \rar& Y^{n-1}\rar\dar& Y^n\rar[two heads]{p}\dar& A\dar[equals]\\
\eta:& A' \rar[tail] & \cdots\rar& X^{n-1}\rar& X^n\rar[two heads] & A
\end{tikzcd}
\end{equation*}
So, by Proposition 4.8 of \cite{Ja}, $\eta$ is the $n$-pushout of $\gamma$ along the morphism $g$. Hence $\eta= \Ext^n(A, g)(\gamma)$ and is a part of the following diagram
\begin{equation*}
\begin{tikzcd}[column sep=small]
& \gamma':&  K\dlar{g}\ar[tail]{rr}{i}\ar[equals]{dd} &&Z^1\dlar\ar{rr}\ar{dd} &&\cdots\dlar[dotted]\ar{rr}\ar[dotted]{dd}&&Z^n\ar{dd} \dlar
\ar[two heads]{rr}&&X\dlar[equals]\ar{dd}{\varphi}\\
\eta': &     A'\ar[tail]{rr}\ar[equals]{dd} &&W^1\ar{rr}\ar{dd}&&\cdots\ar{rr}\ar[dotted]{dd}&& W^n\ar[two heads]{rr}\ar{dd}&&X\ar{dd}[near start]{\varphi}&& \\
&\gamma: ~~~~~~~~~&K\dlar{g}\ar[tail]{rr}&&Y^1\ar{rr}\dlar&&\cdots\dlar[dotted]\ar{rr}&&Y^n\dlar\ar[two heads]{rr}&&A\dlar[equals]\\
\eta:  & A'\ar[tail]{rr}&&X^1\ar{rr}&&\cdots\ar{rr}&&X^n\ar[two heads]{rr}&&A
\end{tikzcd}
\end{equation*}
By assumption, $\gamma'$, the $n$-pullback of $\gamma$ along $\varphi$, is an $\SF$-admissible $n$-exact sequence. Now by Proposition \ref{Ext} we have
\begin{align*}
\ \eta'=&\  \Ext^n(\varphi, A')(\eta)  \\
=& \ \Ext^n(\varphi, A')\Ext^n(A, g)(\gamma) \\
=&  \ \Ext^n(X, g)\Ext^n(\varphi, K)(\gamma) \\
=& \ \Ext^n(X, g)(\gamma').
\end{align*}
So $\eta'$ is the $n$-pushout of $\SF$-admissible $n$-exact sequence $\gamma'$ and hence is an $\SF$-admissible $n$-exact sequence.
\end{proof}

\begin{theorem}\label{17}
Let $(\SC, \SX)$ be an $n$-cluster tilting subcategory with enough projective morphisms and $\SF\subseteq \Ext^n$ be a subfunctor with enough  injective morphisms. Let $A\in\SC$ and consider the $n$-pushout of admissible $n$-exact sequence $\gamma: \ \  K \rightarrowtail Y^1\rt  \cdots  \rt Y^n\stackrel{p} \twoheadrightarrow A$, where $p$ is projective admissible epimorphism, along an $\SF$-injective $\SF$-admissible monomorphism $e: K \rt A'$
\begin{equation*}
\begin{tikzcd}
\gamma :& K\dar{e}\rar[tail]&  Y^1\rar\dar& \cdots \rar&  Y^n\rar[two heads]{p}\dar& A\dar[equals]\\
& A' \rar[tail] & X^{1}\rar & \cdots\rar& X^n\rar[two heads]{\varphi}& A
\end{tikzcd}
\end{equation*}
Then $\varphi: X^n \rt A$ is a special $n$-$\SF$-phantom precover of $A$.
\end{theorem}

\begin{proof}
Since $e\in \Phi(\SF)^{\perp}$, to show that $\varphi$ is a special $n$-$\SF$-phantom precover of $A$, it is enough to show that $\varphi$ is an $n$-$\SF$-phantom morphism. This we do. To this end, by Lemma \ref{16}, it is enough to show that $n$-pullback of $\gamma$ along $\varphi: X^n \rt A$ is an $\SF$-admissible $n$-exact sequence. The diagram
\begin{equation*}
\begin{tikzcd}
\gamma': & K\rar[tail]{i}\dar[equals]& Z^1\rar\dar &\cdots\rar& Z^n\rar[two heads] \dar & X^n\dar{\varphi}\\
\gamma: & K\rar[tail]\dar{e} &Y^1\rar\dar &\cdots\rar& Y^n\rar[two heads]{p}\dar & A\dar[equals] &  \\
 &A' \rar& X^{1}\rar[tail] & \cdots\rar& X^n\rar[two heads]{\varphi}& A
\end{tikzcd}
\end{equation*}
give us, by composition, the following morphism of admissible $n$-exact sequences
\begin{equation*}
\begin{tikzcd}
\gamma': & K\rar[tail]{i}\dar{e}& Z^1\dlar[dotted]{g}\rar\dar &\cdots\rar& Z^n\rar[two heads] \dar& X^n\dlar[dotted]{1_X}\dar{\varphi}\\
 &A' \rar[tail]& X^{1}\rar& \cdots\rar& X^n\rar[two heads]{\varphi}& A
\end{tikzcd}
\end{equation*}
Since there exists morphism $1_X: X \lrt X$ such that $\varphi = \varphi 1_X$, this composition is null-homotopic, see e.g. \cite[Lemma 3.6]{Fe}. So there exists morphism $g: Z^1 \lrt A'$ such that $e=gi$. Now since $e$ is an $\SF$-admissible monomorphism, by Lemma \ref{5}, $i$ is an $\SF$-admissible monomorphism and hence $\gamma'$ is an $\SF$-admissible $n$-exact sequence.
\end{proof}

In particular, we have proved the following theorem that is the main result of this section.

\begin{theorem}\label{1}
Let $(\SC, \SX)$ be an $n$-cluster tilting subcategory of $\SA$ with enough projective morphisms. Then the ideal $\SI$ of $\SC$ is special precovering
if there exists an additive subfunctor $\SF\subseteq\Ext^n$ with enough injective morphisms and $\SI=\Phi(\SF)$.
\end{theorem}

Last result of this section, shows that every special precovering ideal $\SI$, under some conditions, can be represented as an ideal of $n$-$\PB(\SI)$-phantom morphisms.

\begin{theorem}
Let $(\SC, \SX)$ be an $n$-cluster tilting subcategory of $\SA$ with enough $\SX$-injective objects.  Let $\SI$ be an ideal of $\SC$ such that $\SI$ is closed under $n$-coextensions by projectives and $\SI^{\perp}$ is closed under $n$-extensions by injectives. If $\SI$ is a special precovering ideal then there exists an additive subfunctor $\SF\subseteq\Ext^n$ with enough special injective morphisms such that $\SI=\Phi(\SF)$.
\end{theorem}

\begin{proof}
 Since $\SI$ is a special precovering ideal, by Theorem \ref{13}, $(\SI, \SI^{\perp})$ is an  $n$-ideal cotorsion pair. Also since  $\SI$ is closed under $n$-coextensions by projectives and $\SI^{\perp}$ is closed under $n$-extensions by injectives, Salce's Lemma \ref{18}, implies that  the $n$-ideal cotorsion pair $(\SI, \SI^{\perp})$ is complete.  Now put $\SF=\rm{PB}(\SI)$ and use  Corollary \ref{implication 3 to 4}.
 \end{proof}

\section{Wakamatsu's Lemma}\label{Wakamatsu}
In this section we state and prove an ideal version of the Wakamatsu's Lemma \cite{Wak} in higher homological algebra. An ideal version of Wakamatsu's Lemma in an exact category is proved in \cite[Theorem 10.3]{FH}. For a version in $(n+2)$-angulated categories and a version in $n$-abelian categories see \cite[\S 3]{Jor} and
\cite[Theorem 4.2]{AMS}, respectively.

Throughout the section $(\SC, \SX)$ is an $n$-cluster tilting subcategory of an exact category $(\SA, \SE)$.

We need the following version of \cite[Definition-Proposition 2.15]{IJ} in the proof of the main result of this section. The proof follows using similar argument so we skip it. As before, for simplicity of notation, we write $\Ext^i$ instead of $\Ext^i_{\SX}$.

\begin{proposition}\label{Definition-Proposition 2.15}
 Let $(\SC, \SX)$ be an $n$-cluster tilting subcategory of the exact category $(\SA, \SE)$. The following conditions are equivalent.
\begin{itemize}
\item [$(a)$] $\Ext^i(\SC, \SC)= 0$, for all $i \notin n\Z.$
\item [$(b)$] $\SC$ is closed under $n$-syzygies, that is, $\Om^n(\SC) \subseteq \SC.$
\item [$(c)$] $\SC$ is closed under $n$-cosyzygies, that is, $\Om^{-n}(\SC) \subset \SC.$
\item [$(d)$] For each $X \in \SC$ and for each admissible $n$-exact sequence
\[ L \rightarrowtail  M^1 \lrt M^2 \lrt \cdots \lrt  M^n \twoheadrightarrow N  \]
in $\SX$, there exists an exact sequence
\begin{align*}
0 \lrt & \ \ \ \ \SC(X,L) \ \ \lrt  \ \ \SC(X,M^1) \ \  \lrt  \cdots \lrt \ \ \SC(X,M^n) \ \ \  \lrt \ \SC(X,N)\\
{}\lrt & \ \ \Ext^n(X,L) \lrt  \Ext^n(X,M^1) \ \lrt  \cdots \lrt \Ext^n(X,M^n) \ \lrt \Ext^n(X,N)\\
{}\lrt & \ \Ext^{2n}(X,L)  \lrt  \Ext^{2n}(X,M^1) \lrt  \cdots\lrt \Ext^{2n}(X,M^n) \lrt \Ext^{2n}(X,N)\\
{} \lrt & \cdots
\end{align*}
of abelian groups.

\item [$(e)$]  For each $X\in\SC$ and for each admissible $n$-exact sequence
\[ L \rightarrowtail  M^1 \lrt M^2 \lrt \cdots \lrt  M^n \twoheadrightarrow N  \]
in $\SX$, there exists an exact sequence
\begin{align*}
0 \lrt & \ \ \ \ \SC(N,X) \ \ \lrt  \ \ \SC(M^n,X) \ \  \lrt  \cdots \lrt \ \ \SC(M^1,X) \ \ \  \lrt \ \SC(L,X)\\
{}\lrt & \ \ \Ext^n(N,X) \lrt  \Ext^n(M^n,X) \ \lrt  \cdots \lrt \Ext^n(M^1,X) \ \lrt \Ext^n(L,X)\\
{}\lrt & \ \Ext^{2n}(N,X)  \lrt \Ext^{2n}(M^n,X) \lrt  \cdots\lrt \Ext^{2n}(M^1,X) \lrt \Ext^{2n}(L,X)\\
{} \lrt & \cdots
\end{align*}
of abelian groups.
\end{itemize}
\end{proposition}

Recall that an object $A$ of $\SC$ is said to be in $\SI$ if the identity morphism $1_A$ is in $\SI(A, A)$.

\begin{definition}
Let $\SI$ be an ideal of $\SC$. We say that $\SI$ is left closed under $n$-extensions by objects of $\SI$, if for every morphism of $\SX$-admissible $n$-exact sequences in $\SC$ such as
\begin{equation*}
\begin{tikzcd}
X^0 \dar{f^0} \rar[tail] & X^1 \rar \dar{f^1} & X^2 \rar\dar{f^2} & \cdots \rar & X^n \rar[two heads] \dar{f^n} &  X^{n+1} \dar[equals]\\
Y^0 \rar[tail] & Y^1 \rar & Y^2 \rar & \cdots \rar & Y^n \rar[two heads]  & X^{n+1}
\end{tikzcd}
\end{equation*}
$X^1$ is an object of $\SI$ if $f^0 \in \SI$ and $X^{n+1}$ is an object of $\SI$.
\end{definition}

Finally let us recall the definition of an $\SI$-cover.

\begin{definition}
Let $\SI$ be an ideal of $\SC$ and $A$ be an object of $\SC$. An $\SI$-precover $i: C \lrt A$ of $A$ is called an $\SI$-cover if every endomorphism $f: C\lrt C$ with the property that $if=i$ is necessarily an automorphism.
\end{definition}

Now we have the necessary ingredients to state and prove the main result of this section. For a version in $(n+2)$-angulated categories see Lemma 3.1 of \cite{Jor}.

\begin{theorem}(Wakamatsu's Lemma)
Let $(\SC, \SX)$ be an $n$-cluster tilting subcategory of an exact category $(\SA, \SE)$ with enough $\SX$-injective objects. Let $\CI$ be an ideal of $\SC$ which is left closed under $n$-extensions by objects in $\SI$.  Let $A$ be an object of $\SC$ and $i: I \lrt A$ be the $\SI$-cover of $A$. Then for every $X\in\SI$, there exists the exact sequence
\[0\lrt \Ext^{n}(X, {K_n}) \lrt \Ext^{n}(X, {K_{n-1}}) \lrt\cdots \lrt \Ext^{n}(X, {K_1}) \lrt 0,\]
of abelian groups, where $K_n\lrt K_{n-1} \lrt \cdots \lrt K_1$  is an $n$-kernel of $i$.
\end{theorem}

\begin{proof}
Consider $n$-exact sequence
\[K_n \rightarrowtail K_{n-1} \lrt \cdots \lrt K_1 \lrt I \twoheadrightarrow A\]
in $\SX$. Since $\SC$ has enough $\SX$-injective objects, it is closed under $n$-cosyzygies and hence by equivalent conditions of Proposition \ref{Definition-Proposition 2.15}, we have the following long exact sequence
{\footnotesize{
\[\begin{tikzcd}
&&&\cdots\rar&\SC(X, I) \rar{i^*} &\SC(X, A)\rar&{} \\  {}\rar&\Ext^n(X, K_n)\rar &\Ext^n(X, K_{n-1}) \rar &\cdots \rar &\Ext^n(X, I)\rar{\widehat{i}} & \Ext^n(X, A)\rar & \cdots \end{tikzcd}\]}}
\noindent of abelian groups. To prove the theorem, it is enough to show that $i^*$  is surjective and $\widehat{i}$ is injective. This we do. Since $X \in \Ob(\SI)$ and $i$ is an $\SI$-cover, it follows by definition that $i^*$ is surjective. So it remains to prove that $\widehat{i}$ is injective. This follows from \cite[Lemma 5.1]{Fe}. We reproduce the proof here. Let
\[\eta: I \rightarrowtail X^1 \lrt X^2 \lrt \cdots \lrt X^n \twoheadrightarrow X\]
be an element of $\Ext^n(X, I)$ that maps to zero in $\Ext^n(X, A)$, that is, the $n$-pushout of $\eta$ along $i$, say $\eta'$, is a contractible $n$-exact sequence. We show that $\eta$ itself should be a contractible $n$-exact sequence. To see this consider the following $n$-pushout diagram
\begin{equation*}
\begin{tikzcd}
\eta\dar{\widehat{i}}: &I \dar{i} \rar[tail]{d^0_X} & X^1 \rar{d^1_X} \dar{f^1} \dlar[dotted] & X^2 \rar{d^2_X}\dar{f^2} & \cdots \rar & X^n \rar[two heads]{d^n_X} \dar{f^n} &  X \dar[equals] \dlar[dotted]\\
\eta': &A \rar[tail]{d^0_Y} & Y^1 \rar{d^1_Y} & Y^2 \rar{d^2_Y} & \cdots \rar & Y^n \rar[two heads]{d^n_Y}  & X
\end{tikzcd}
\end{equation*}
To show that $\eta$ is contractible, by \cite[Proposition 2.6]{Ja}, it is enough to show that $d^0_X$ is a split monomorphism. Since $\eta'$ is contractible, there exists morphism $s^{n+1}: X \rt Y^n$ such that $d^n_Y s^{n+1}=1_X$. This, in view of Lemma 3.6 of \cite{Fe} implies that $\widehat{i}: \eta \lrt \eta'$ should be null-homotopic.   So, in particular, there exists a morphism $s^1: X^1 \lrt A$ such that $s^1d^0_X = i$.

Now since $\SI$ is left closed under $n$-extensions by objects of $\SI$ and $i\in \SI$ and $X$ is an object in $\SI$, $X^1 \in \Ob(\SI)$. So $s^1: X^1 \lrt A$ is a morphism in $\SI$ and hence factors through $i$. Let $\alpha: X^1 \lrt I$ be such that $i\alpha= s^1$.  By composing with $d^0_X$  we get $i\alpha d^0_X=s^0d^0_X=i$. Since $i$ is an $\SI$-cover, $\alpha d^0_X: I\lrt I$ is an isomorphism. So $d^0_X$ is a split monomorphism. Hence $\eta$ is contractible, as it was desired.
\end{proof}

 \end{document}